\documentclass{article}
\pdfoutput=1
\usepackage{amsmath,amssymb,amsthm} 
\usepackage{graphicx} 
\usepackage[numbers,sort&compress]{natbib} 
\usepackage[margin=1in]{geometry}

\usepackage[utf8]{inputenc} 
\usepackage[T1]{fontenc}    
\usepackage{hyperref}       
\usepackage{url}            
\usepackage{booktabs}       
\usepackage{amsfonts}       
\usepackage{nicefrac}       
\usepackage{microtype}      

\usepackage{algorithm}
\usepackage{algcompatible}
\usepackage{amsmath}
\usepackage{amsthm}
\usepackage{amssymb}
\usepackage{algpseudocode}
\usepackage{mathtools}
\usepackage{bm}
\usepackage{doi}
\usepackage{graphicx}
\usepackage{color}
\usepackage{authblk}
\usepackage[toc,page]{appendix}

\hypersetup{hidelinks=true}

\usepackage{scalerel}

\makeatletter
\newcommand*\bigcdot{\mathpalette\bigcdot@{.5}}
\newcommand*\bigcdot@[2]{\mathbin{\vcenter{\hbox{\scalebox{#2}{$\m@th#1\bullet$}}}}}
\makeatother



\title{A tree-based radial basis function method for noisy parallel surrogate optimization}

\author[1]{Chenchao Shou\thanks{cshou3@illinois.edu}}
\author[1]{Matthew West\thanks{mwest@illinois.edu}}
\affil[1]{Department of Mechanical Science and Engineering, University of Illinois at Urbana-Champaign}
\date{}

\begin{document}

\newtheorem{defn}{Definition}[section]

\newtheorem{thm}{Theorem}[subsection]
\renewcommand{\thethm}{\arabic{thm}}

\newtheorem{lem}{Lemma}[section]
\newtheorem{col}{Corollary}[section]
\newtheorem{assumption}{Assumption}[section]
\newtheorem*{remark}{Remark}
\newtheorem*{note}{Note}
\newtheorem{assumpa}{Assumption}
\renewcommand\theassumpa{A.\arabic{assumpa}}
\newtheorem{assumpb}{Assumption}
\renewcommand\theassumpb{B.\arabic{assumpb}}

\maketitle

\begin{abstract}
Parallel surrogate optimization algorithms have proven to be efficient methods for solving expensive noisy optimization problems. In this work we develop a new parallel surrogate optimization algorithm (ProSRS), using a novel tree-based ``zoom strategy'' to improve the efficiency of the algorithm. We prove that if ProSRS is run for sufficiently long, with probability converging to one there will be at least one point among all the evaluations that will be arbitrarily close to the global minimum. We compare our algorithm to several state-of-the-art Bayesian optimization algorithms on a suite of standard benchmark functions and two real machine learning hyperparameter-tuning problems. We find that our algorithm not only achieves significantly faster optimization convergence, but is also 1--4 orders of magnitude cheaper in computational cost.
\end{abstract}

\section{Introduction}


We consider a general global noisy optimization problem:
\begin{equation}\label{eq:optim_prob}
\begin{gathered}
\text{minimize}\quad F(x), \quad \text{where } F(x) \coloneqq \mathbb{E}_\omega[f(x,\omega)],\\
\text{s.t.}\quad x\in \mathcal{D} = [a_1,b_1]\times [a_2,b_2]\times\ldots\times [a_d,b_d] \subseteq \mathbb{R}^d,
\end{gathered}
\end{equation}
where $f$ is an expensive black-box function, $\omega$ captures noise (randomness) in the function evaluation and the dimension $d$ is low to medium (up to tens of dimensions). We assume that only noisy evaluations $f$ are observed and the underlying objective function $F$ is unknown. The problem in Eq.~\ref{eq:optim_prob} is a standard optimization problem \citep{Amaran2016,rakshit2016noisy} that appears in many applications including operations \citep{kochel2005simulation}, engineering designs \citep{prakash2008design}, biology \citep{xie2012optimization,romero2013navigating}, transportation \citep{chen2014surrogate,osorio2010simulation} and machine learning \citep{snoek2012practical,wu2016parallel}.

Another problem relevant to Eq.~\ref{eq:optim_prob}  is a stochastic bandit with infinitely many arms \citep{wang2009algorithms,bubeck2011x,carpentier2015simple,li2017infinitely}. In this type of problem, the goal is to find the optimal strategy within a continuous space so that the expected cumulative reward is maximized. Compared to the problem considered in this work, the objective function in a bandit problem is typically not expensive. As a result, the solution strategies for these two types of problems are generally different.

Surrogate-based optimization algorithms are often used to solve the expensive problem in Eq.~\ref{eq:optim_prob} \citep{Amaran2016}. Because the algorithm exploits the shape of the underlying objective function, surrogate-based optimization can often make good progress towards the optimum using relatively few function evaluations compared to derivative-free algorithms such as Nelder-Mead and Direct Search algorithms \citep{conn1997recent,regis2007stochastic}. Generally speaking, this method works as follows: in each iteration, a surrogate function that approximates the objective $F$ is first constructed using available evaluations, and then a new set of point(s) is carefully proposed for the next iteration based on the surrogate. Because the function $f$ is expensive, spending extra computation in determining which points to evaluate is often worthwhile.

Within the family of surrogate-based optimization methods, parallel surrogate optimization algorithms propose multiple points in each iteration, and the expensive evaluations of these points are performed in parallel \citep{haftka2016parallel}. Compared to the serial counterpart, parallel surrogate optimization uses the parallel computing resources more efficiently, thereby achieving better progress per unit wall time.

A popular method for noisy parallel surrogate optimization is Bayesian optimization \citep{snoek2012practical,contal2013parallel,desautels2014parallelizing,shah2015parallel,azimi2010batch,gonzalez2016batch,wu2016parallel}. Bayesian optimization typically works by assuming a Gaussian process prior over the objective function, constructing a Gaussian process (GP) surrogate \citep{rasmussen2006gaussian} with the evaluations, and proposing new points through optimizing an acquisition function. Common acquisition functions are expected improvement (EI) \citep{snoek2012practical}, upper confidence bound (UCB) or lower confidence bound (LCB) \citep{contal2013parallel,desautels2014parallelizing}, and information-theoretic based \citep{shah2015parallel,wu2016parallel}.

One issue with Bayesian methods is the high computational cost. Typically, training a GP surrogate requires solving a maximum likelihood problem, for which operations of complexity proportional to the cube of the number of evaluations are performed for many times \citep{quinonero2005analysis}. To propose new points, Bayesian optimization usually requires the solution of sub-optimization problems (e.g., maximizing expected improvement) with the possible use of Monte Carlo procedures \citep{snoek2012practical,azimi2010batch}. When many points are evaluated per iteration, so that the number of evaluations accumulates quickly with the number of iterations, Bayesian optimization algorithm itself can be even more expensive than the evaluation of the function $f$, and this is indeed observed in real hyperparameter-tuning problems (see Section~\ref{subsec:optim_alg_cost}).

In this work, we develop a novel algorithm called ProSRS for noisy parallel surrogate optimization. Unlike Bayesian optimization that uses a GP model, our algorithm uses a radial basis function (RBF), which is more efficient computationally. We adopt an efficient framework, known as stochastic response surface (SRS) method \citep{regis2007stochastic,regis2009parallel}, for proposing new points in each iteration. The sub-optimization problems in the SRS method are discrete minimization problems. Compared to the original parallel SRS work \citep{regis2009parallel}, our work: (1) introduces a new tree-based technique, known as the ``zoom strategy'', for efficiency improvement, (2) extends the original work to the noisy setting (i.e., an objective function corrupted with random noise) through the development of a radial basis regression procedure, (3) introduces weighting to the regression to enhance exploitation, (4) implements a new SRS combining the two types of candidate points that were originally proposed in SRS \citep{regis2007stochastic}.

We compare our algorithm to three well-established parallel Bayesian optimization algorithms. We find that our algorithm shows superior optimization performance on both benchmark problems and real hyperparameter-tuning problems, and yet its cost is orders of magnitude lower. The fact that our algorithm is significantly cheaper means that our algorithm is suitable for a wider range of optimization problems, not just very expensive ones.

The remainder of this paper is organized as follows. In Section~\ref{sect:algorithm}, we present our algorithm. In Section~\ref{sect:converge}, we show a theoretical convergence result. We then demonstrate numerical results in Section~\ref{sect:numerical} and finally conclude in Section~\ref{sect:conclusion}.

\section{The ProSRS algorithm}\label{sect:algorithm}

Conventional surrogate optimization algorithms use all the expensive function evaluations from past iterations to construct the surrogate. As the number of evaluations grows over iterations, the cost of conventional methods thus increases. Indeed, the cost can increase rather quickly with the number of iterations, especially when a large number of points are evaluated in parallel per iteration.

To overcome this limitation, we develop a novel algorithm that does not necessarily use all the past evaluations while still being able to achieve good optimization performance. The key intuition here is that once an optimal region is approximately located, progress can be made by focusing on the evaluations within this region. This idea is illustrated in Fig.~\ref{fig:zoom_diagram}, where the red curve is a surrogate built with all the evaluations. Now suppose we restrict the domain to a smaller region as indicated by the dashed black box and only fit the evaluation data within that region. We still obtain a good surrogate (blue curve) around the optimum, and it is cheaper as we are using fewer evaluations to do so. We now proceed with our optimization, treating the restricted region as our new domain and the local fit as our surrogate for optimization. This idea of recursively optimizing over hierarchical domains lies at the heart of our algorithm. In this paper, we call this technique the ``zoom strategy''. Because it requires less evaluation data to build a local surrogate than to build a global one, the zoom strategy can significantly reduce the cost of the algorithm.

For ease of describing the relationships between different domains, we introduce the notion of a node. A node consists of a domain together with all the information needed by an optimization algorithm to make progress for that domain. We call the process of restricting the domain to a smaller domain the ``zoom-in'' process, in which case the node associated with the original domain is a ``parent'' node and the node for the restricted domain is a ``child'' node. The reverse process of zooming in is referred to as the ``zoom-out'' process (i.e., the transition from a child node to its parent node). See Fig.~\ref{fig:tree_diagram} for an illustration of this structure.

\begin{figure}
\begin{center}
\includegraphics[width=0.9\textwidth]{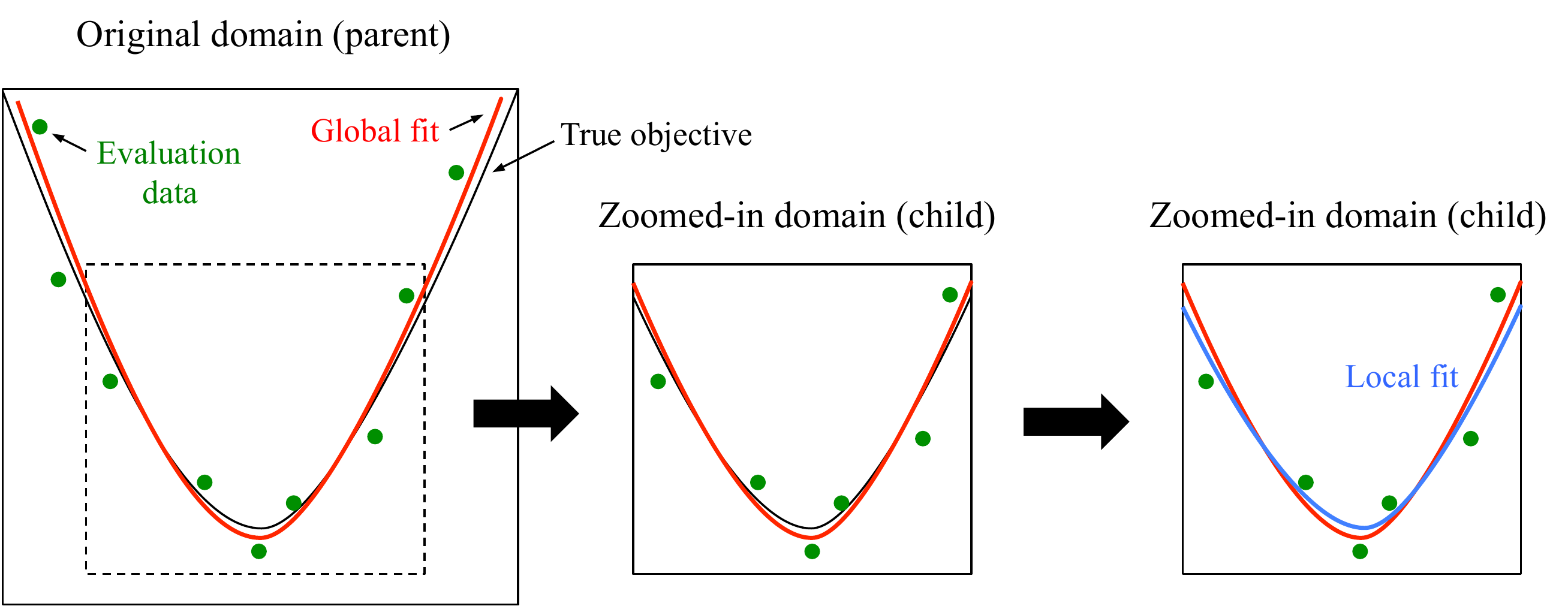}
\caption{Illustration of the zoom strategy on a 1-D parabola. The red curve shows the surrogate fit to all the noisy evaluations (green dots) of the objective function (black curve). The blue curve shows the surrogate fit using only the local evaluation data in the zoomed-in domain. The local fit is likely to agree well with the global fit on the restricted domain, and is much cheaper to construct.}
\label{fig:zoom_diagram}
\end{center}
\end{figure}

\begin{figure}
\begin{center}
\includegraphics[width=0.6\textwidth]{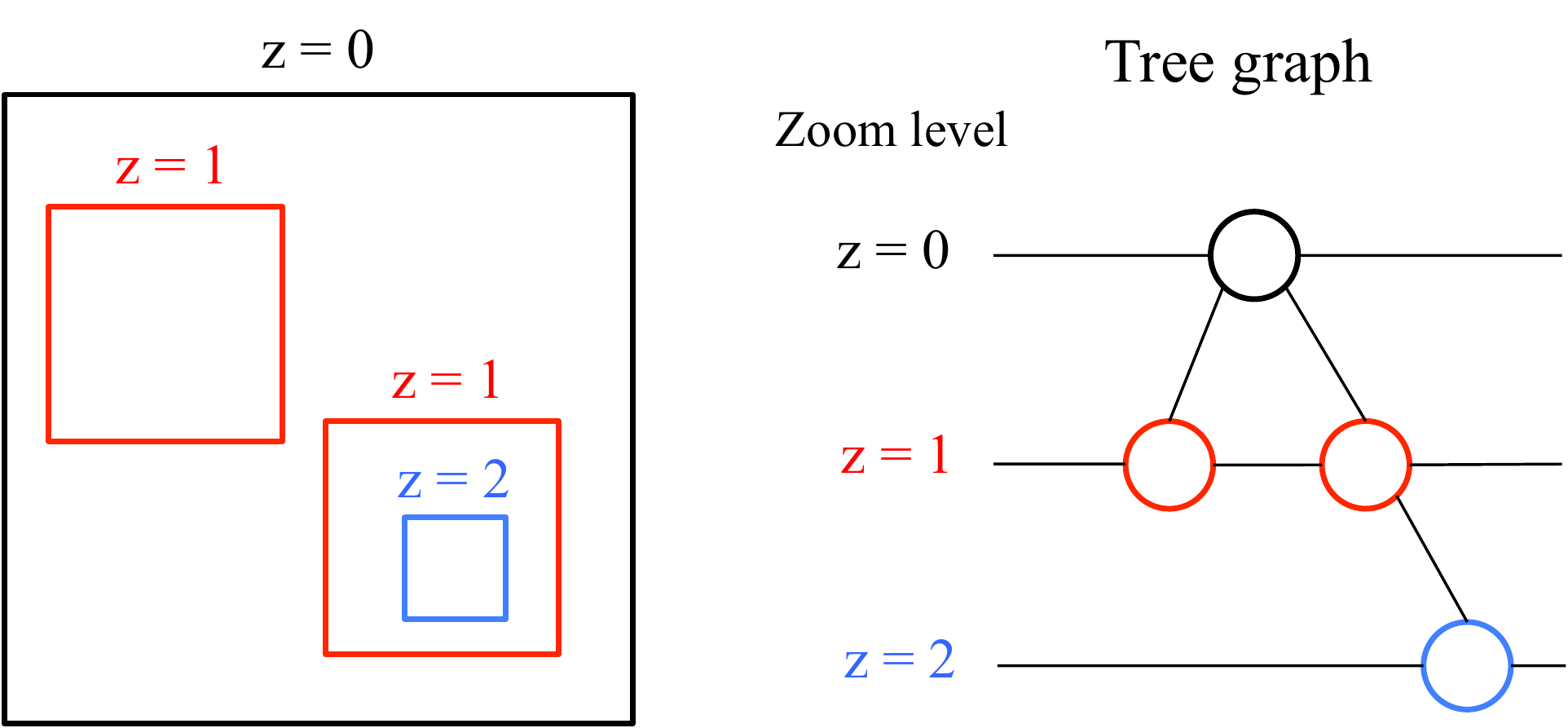}
\caption{Illustration of the tree structure of ProSRS algorithm on a 2-D problem. The black box on the left shows the domain of a root node. The two red boxes and one blue box show two children and one grandchild of the root node.}
\label{fig:tree_diagram}
\end{center}
\end{figure}

\subsection{Overview of the algorithm}

We now present our algorithm, namely Progressive Stochastic Response Surface (ProSRS)\footnote{Code is publicly available at \url{https://github.com/compdyn/ProSRS}.}, in Alg.~\ref{alg:ProSRS}. Like most surrogate optimization algorithms, ProSRS starts with a space-filling design of experiments (DOE). Here we use Latin hypercube sampling with maximin criterion for the initial design. In our algorithm, a node $\mathcal{N}$ is formally defined by a quadruplet:
\begin{equation}
  \label{eqn:node}
  \mathcal{N} = (D, \Omega, S, \beta),
\end{equation}
where $D$ is the evaluation data in the domain $\Omega$. The variable $S$ characterizes the exploitation (versus exploration) strength of ProSRS. Mathematically, it is a tuple:
\begin{equation}
  \label{eqn:state}
  S = (\gamma,p,\sigma),
\end{equation}
where $\gamma$ is a radial basis regression parameter (see Section~\ref{sub_sect:rbf}) and $p,\sigma$ are two parameters in the step of proposing new points (see Section~\ref{sub_sect:srs}). The variable $\beta$ in Eq.~\ref{eqn:node} is the zoom-out probability.

\begin{algorithm}
\caption{Progressive Stochastic Response Surface (ProSRS)}\label{alg:ProSRS}
\begin{algorithmic}[1]
\State \textbf{Inputs:} $m$, $\beta_\text{init}$, $S_\text{init}$ and $N$
\State Generate $m$ Latin hypercube samples: $X = (x_1, x_2,\ldots,x_m)$
\State Evaluate samples $X$ in parallel to give $Y = (y_1, y_2,\ldots,y_m)$
\State Initialize the current node = $(D,\Omega,\beta,S)$ with evaluation data $D = (X,Y)$, domain $\Omega =$ optimization domain $\mathcal{D}$, zoom-out probability $\beta = \beta_\text{init}$ and variable $S = S_\text{init}$
\For {$\text{iteration} = 1,2,\ldots,N$}
	\State Obtain $D,\Omega,\beta,S$ from the current node
	\State $g \gets \text{RBF}(D, S)$ \label{pseudo_code:rbf}\Comment{Build radial basis surrogate (see Sect.~\ref{sub_sect:rbf})}
	\State $X_\text{new}\gets \text{SRS}(D,\Omega,S,g)$\label{pseudo_code:srs} \Comment{Propose new points (see Sect.~\ref{sub_sect:srs})}
	\State $Y_\text{new} \gets$ evaluate points $X_\text{new}$ in parallel
	\State Augment evaluation data $D$ with $(X_\text{new},Y_\text{new})$
	\State Update the variable $S$ of current node \label{pseudo_code:update_state} \Comment{see Sect.~\ref{sub_sect:state}}
	\If {$S$ reaches the critical value} \label{pseudo_code:crit_state}
		\If {restart condition is met} \label{pseudo_code: restart_condition}
			\State Restart from DOE \label{pseudo_code:restart}
		\Else
			\State Create or update a child node \label{pseudo_code:zoom_in}\Comment{Zoom in (see Sect.~\ref{sub_sect:zoom})}
			\State Reset variable $S$ of current node, and set the child node to be the current node
		\EndIf
	\EndIf
	\If {no restart \textbf{and} the parent of current node exists} \label{pseudo_code:check_zoom_out}
		\State With probability $\beta$, set its parent node to be the current node \Comment{Zoom out}\label{pseudo_code:zoom_out}
	\EndIf
\EndFor
\State \textbf{return} $x_\text{best} =$ the evaluated point with the lowest $y$ value
\end{algorithmic}
\end{algorithm}

For each iteration, we first construct a radial-basis surrogate using the evaluation data $D$ (Line~\ref{pseudo_code:rbf}), followed by the step of proposing new points for parallel evaluation (Line~\ref{pseudo_code:srs}). The proposed points must not only exploit the optimal locations of a surrogate, but also explore the untapped regions in the domain to improve the quality of the surrogate. Indeed, achieving the appropriate balance between exploitation and exploration is the key to the success of a surrogate optimization algorithm. For this, we use an efficient procedure, known as Stochastic Response Surface (SRS) method, that was first developed by \citet{regis2007stochastic} and later extended to the parallel setting in their subsequent work \citep{regis2009parallel}.

After performing expensive evaluations in parallel, we update the exploitation strength variable $S$ (Line~\ref{pseudo_code:update_state}) so that for a specific node, the exploitation strength progressively increases with the number of iterations (see Section~\ref{sub_sect:state} for the update rule). The purpose of this step is to help locate the optimal region for zooming in. Once the exploitation strength reaches some prescribed threshold (Line~\ref{pseudo_code:crit_state}; see Section~\ref{sub_sect:zoom} for details), the algorithm will decide to zoom in (Line~\ref{pseudo_code:zoom_in}) by setting a child to be the current node (neglecting the restart step in Line~\ref{pseudo_code:restart} for now). The updating of the variable $S$ and the zoom-in mechanism generally make ProSRS ``greedier'' as the number of iterations increases. To balance out this increasing greediness over iterations, we implement a simple $\epsilon$-greedy policy by allowing the algorithm to zoom out with some small probability in each iteration (Line~\ref{pseudo_code:zoom_out}). Because of the mechanism of zooming in and out, ProSRS will generally form a ``tree'' during the optimization process, as illustrated in Fig.~\ref{fig:tree_diagram}.

Finally we would like to address the restart steps (Line~\ref{pseudo_code: restart_condition} and \ref{pseudo_code:restart}) in Alg.~\ref{alg:ProSRS}. We make the algorithm restart completely from scratch when it reaches some prescribed resolution after several rounds of zooming in. Specifically, to check whether to restart, we first perform the step of creating or updating a child node like the normal zoom-in process (Line~\ref{pseudo_code:zoom_in}). Suppose the resulted child node has $n$ points in its domain $\Omega\subseteq\mathbb{R}^d$, then ProSRS will restart if for all $i = 1,2,\ldots,d$,
\begin{equation}\label{eq:restart_cond}
n^{-\frac{1}{d}}\ell_i(\Omega) < r(b_i-a_i),
\end{equation}
where $r\in (0,1)$ is a prescribed resolution parameter, $\ell_i(\Omega)$ denotes the length of the domain $\Omega$ in the $i^\text{th}$ dimension, $a_i$ and $b_i$ are the bounds for the optimization domain $\mathcal{D}$ (Eq.~\ref{eq:optim_prob}). The reason for restarting from a DOE is to avoid the new runs being biased by the old runs so that the algorithm has a better chance to discover other potentially optimal regions. Indeed, extensive study \citep{regis2007stochastic,regis2009parallel,regis2016trust} has shown that restarting from the initial DOE is better than continuing the algorithm with past evaluations.

\subsection{Weighted radial basis surrogate functions}\label{sub_sect:rbf}

Given the evaluation data $D = \{(x_1,y_1),(x_2,y_2),\ldots,(x_n,y_n)\}$, a radial basis surrogate takes the form
\begin{equation}\label{eq:rbf_form}
	g(x) = \sum_{i=1}^n c_i\phi(\lVert x-x_i\rVert), \quad x\in\mathbb{R}^d,
\end{equation}
where the function $\phi$ is a radial basis function. In this work, we choose $\phi$ to be a multiquadric function. The radial basis coefficients $c_i$ are obtained by minimizing the $\text{L}_2$-regularized weighted square loss:
\begin{equation}\label{eq:rbf_loss}
	\text{Loss} = \sum_{j=1}^n e^{\gamma\hat{y}_j}\big(y_j-g(x_j)\big)^2+\lambda\sum_{j=1}^n c_j^2, \quad \text{with} \quad \hat{y}_j = \frac{y_j-\min y_k}{\max y_k-\min y_k},
\end{equation}
where $\gamma$ is a non-positive weight parameter (one component of the variable $S$; see Eq.~\ref{eqn:state}) and $\lambda$ is a regularization constant determined automatically through cross validation. This loss function is quadratic in the coefficients $c_i$ so that the minimization problem admits a unique solution and can be solved efficiently.

The term $e^{\gamma\hat{y}_j}$ in Eq.~\ref{eq:rbf_loss} represents the weight for the $j^\text{th}$ data point, and $\hat{y}_j$ can be interpreted as the normalized $y$ value with the understanding that $\hat{y}_j = 0$ if $\max y_k = \min y_k$. It is clear that $\gamma = 0$ disables the weighting in the RBF regression. When $\gamma$ is negative, the points with smaller $y$ values gain more weight, so the RBF regression produces a better fit for the points with low $y$ values (the ``best'' points). Consequently, smaller weight parameter $\gamma$ values imply greater exploitation.

\subsection{Stochastic response surface method}\label{sub_sect:srs}

To propose new points for parallel evaluations, we use the general Stochastic Response Surface (SRS) framework \citep{regis2007stochastic,regis2009parallel}. The first step of the stochastic response surface method is to randomly generate candidate points in the domain $\Omega$. In the original SRS work \citep{regis2007stochastic}, the authors introduced two types of candidate points and proposed one algorithm for each type. Here we consider the candidate points to be a mixture of both types. 

Type~I candidate points are sampled uniformly over the domain. Type~II candidate points are generated by adding Gaussian perturbations around the current best point $x^*$, where $x^*$ is the point in the evaluation data $D$ with the lowest value of the RBF surrogate $g$. The covariance matrix for the Gaussian perturbation is a diagonal matrix with its diagonal being $\sigma^2 l_i^2(\Omega)$ ($i=1,2,\ldots,d$), where $\sigma$ is one component of the variable $S$ (see Eq.~\ref{eqn:state}) and $l_i(\Omega)$ is the length of the domain in the $i^\text{th}$ dimension. Any generated point that lies outside the domain would be replaced by the nearest point in the domain so that all the Type~II candidate points are within $\Omega$. The proportion of these two types of candidate points is controlled by a parameter $p$, which is another component of the variable $S$. Specifically, we generate $1000d$ candidate points with a fraction of $\frac{1}{10}\lfloor 10p\rfloor$ points being Type~I and the remainder being Type~II.

The second step is to measure the quality of each candidate point using two criteria: the value of the response surface (RBF surrogate) and the minimum distance from previously evaluated points. The points with low response values are of high exploitation value, while the ones with large minimum distances are of high exploration value. In the SRS method, every candidate point is given a score on each of the two criteria, and a weight between 0 and 1 is used for trading off one criterion for the other. For our algorithm, we generate an array of weights that are equally-spaced in the interval $[0.3, 1]$ with the number of weights being equal to the number of proposed points per iteration (if the number of proposed points per iteration is one, we alternate weights between 0.3 and 1 from iteration to iteration). This weight array, also known as the ``weight pattern'' in the original work \citep{regis2007stochastic}, is used to balance between exploitation and exploration among the proposed points. The procedures of scoring the candidate points and selecting the proposed points from the candidate points based on the weight pattern are described in detail in \citet{regis2009parallel}.

\subsection{Update procedure for variable $S$}\label{sub_sect:state}

After obtaining new evaluations, we update the variable $S$ of the current node (Line~\ref{pseudo_code:update_state} of Alg.~\ref{alg:ProSRS}). The goal of this updating step is to gradually increase the exploitation strength. As listed in Eq.~\ref{eqn:state}, the variable $S$ of a node consists of 3 parameters: (1) a weight parameter $\gamma$ for radial basis regression, (2) a parameter $p$ that controls the proportion of Type~I candidate points in the SRS method, and (3) a parameter $\sigma$ that determines the spread of Type~II candidate points. The exploitation strength will be enhanced by decreasing any of these 3 parameters.

\begin{algorithm}
\caption{Update $p$, $\sigma$ and $\gamma$}\label{alg:update_state}
\begin{algorithmic}
	\If{$p\geq 0.1$}
		\State $p \gets p n_\text{eff}^{-\frac{1}{d}}$
	\ElsIf{the counter for number of consecutive failed iterations = $C_\text{fail}$}
		\State Reset the counter
		\State $\sigma \gets \sigma/2$ and $\gamma \gets \gamma-\Delta\gamma$
	\EndIf
\end{algorithmic}
\end{algorithm}

The update rule is specified in Alg.~\ref{alg:update_state}, which can be understood as having two separate phases. The first phase is when there are still some Type~I candidate points generated in the SRS method (i.e., $p \geq 0.1$). During this phase, the values of $\sigma$ and $\gamma$ are unchanged but the $p$ value is decreased with each iteration. The rate of decrease is determined by $n_\text{eff}^{-1/d}$, where $n_\text{eff}$ is the effective number of evaluations for the current iteration. The effective number of evaluations $n_\text{eff}$ is computed by first uniformly partitioning the domain $\Omega$ into cells with the number of cells per dimension being equal to $\lceil n^{1/d} \rceil$, where $n$ is the number of points in the evaluation data $D$. Then $n_\text{eff}$ is number of cells that are occupied by at least one point. The quantity $n_\text{eff}^{1/d}$ can be viewed as a measurement of the density of the evaluated points in the domain $\Omega$. Therefore, we essentially make the decreasing rate proportional to the evaluation density.

When the $p$ value drops below 0.1, so that all the candidate points are Type~II, we enter the second phase of the state transition, where the parameter $p$ does not change but $\sigma$ and $\gamma$ are reduced. Just like in \citet{regis2007stochastic}, we use the number of consecutive failures as the condition for deciding when to reduce the value of $\sigma$. Here an iteration is counted as a failure if the best $y$ value of the proposed points for the current iteration does not improve the best $y$ value of the evaluations prior to the proposing step. The counter is set to zero at the beginning of the algorithm, and starts to count the number of consecutive failures only when $p<0.1$. Whenever the number of consecutive failures reaches some prescribed threshold $C_\text{fail}$, we reduce $\sigma$ by half and decrease $\gamma$ by $\Delta\gamma$.

\subsection{Zoom strategy}\label{sub_sect:zoom}

The updating of the variable $S$ (Line~\ref{pseudo_code:update_state}) will make the parameter $\sigma$ gradually decrease over iterations. Once $\sigma$ drops below some critical value $\sigma_\text{crit}$ (i.e., $S$ reaches the critical value in Line~\ref{pseudo_code:crit_state}) and the restart condition is not satisfied, the algorithm will zoom in by either creating a new child node or updating an existing child node. Specifically, we start the zoom-in process by finding the point that has the lowest fit value among the evaluation data $D$, which we will denote as $x^*$. Depending on the location of $x^*$ and the locations of the children of the current node, there are two possible scenarios.

The first scenario is that $x^*$ does not belong to the domain of any of the existing child nodes or there is no child for the current node. In this case a new child node is created. The domain $\Omega$ of this child node is generated by shrinking the domain of the current node with the center being at $x^*$ and the length of each dimension being $\rho$-fractional of that of the current domain. The parameter $\rho\in (0,1)$ is called the zoom-in factor, which is a constant set prior to the start of the algorithm. After shrinkage, any part of the new domain that is outside the current domain will be clipped off so that the domain of a child is always contained by that of its parent. Given the domain of the new child node, its evaluation data $D$ are all the past evaluations that are within this domain. The zoom-out probability $\beta$ and the variable $S$ of this child node are set to the initial values $\beta_\text{init}$ and $S_\text{init}$ respectively.

The other possibility is that $x^*$ belongs to at least one child of the current node. Among all children whose domains contain $x^*$, we select the child whose domain center is closest to $x^*$. The evaluation data $D$ of this selected child node is updated by including all the past evaluations that are within its domain. Since the selected child node is being revisited, we reduce its zoom-out probability by $\beta \gets \text{max}(\beta/2, \beta_\text{min})$, where $\beta_\text{min}$ is a constant lower bound for the zoom-out probability.

\section{Convergence}\label{sect:converge}

In this section we state a convergence theorem for our ProSRS algorithm (Alg.~\ref{alg:ProSRS}). More specifically, if ProSRS is run for sufficiently long, with probability converging to one there will be at least one point among all the evaluations that will be arbitrarily close to the global minimizer of the objective function. Because the point returned in each iteration is the one with the lowest noisy evaluation (not necessarily with the lowest expected value), as the underlying expectation function is generally unknown, this theoretical result does not immediately imply the convergence of our algorithm. However, in practice one may implement posterior selection procedures for choosing the true best point from the evaluations using discrete optimization algorithms such as those by \citet{nelson2001simple} and \citet{ni2013ranking}.

\begin{thm}\label{thm:converge}
Suppose the objective function $F$ in Eq.~\ref{eq:optim_prob} is continuous on the domain $\mathcal{D}\subseteq\mathbb{R}^d$ and $x_\text{opt}$ is the unique minimizer of $F$, characterized by\footnote{Here we adopt the convention that if $\{x\in\mathcal{D}, \lVert x-x_\text{opt} \rVert \geq \eta\}= \emptyset$, then $\inf_{x\in\mathcal{D}, \lVert x-x_\text{opt} \rVert \geq \eta} F(x) = +\infty$.} $F(x_\text{opt}) = \inf_{x\in\mathcal{D}} F(x) \in (-\infty, +\infty)$ and $\inf_{x\in\mathcal{D}, \lVert x-x_\text{opt} \rVert \geq \eta} F(x) > F(x_\text{opt})$ for all $\eta>0$. Let $x_n$ be the point with the minimum objective value among all the evaluated points up to iteration $n$. Then $x_n \xrightarrow{} x_\text{opt}$ almost surely as $n\to\infty$. 
\end{thm}

\begin{proof}
	See Appendix~\ref{sect:converge_proof}.
\end{proof}

\section{Numerical results}\label{sect:numerical}

We compare our algorithm to three state-of-the-art parallel Bayesian optimization algorithms: GP-EI-MCMC \citep{snoek2012practical}, GP-LP \citep{gonzalez2016batch} with acquisitions LCB and EI. The parameter values of ProSRS algorithm are listed in Table~\ref{table:param_prosrs}, where $d$ is optimization dimension and $N_\text{par}$ is the number of points evaluated in parallel per iteration.

\begin{table}
  \caption{Parameter values for the ProSRS algorithm}
  \label{table:param_prosrs}
  \centering
  \begin{tabular}{ccc}
    	\toprule
    Parameter  &  Meaning & Value \\
    \midrule
    $m$ & number of DOE samples & $\lceil 3/N_\text{par}\rceil N_\text{par}$ \\
    $S_\text{init}$ & initial value of exploitation strength variable $S$ & (0, 1, 0.1) \\
    $\sigma_\text{crit}$ & critical $\sigma$ value & 0.025 \\
    $\beta_\text{init}$ & initial zoom-out probability & 0.02 \\
    $\beta_\text{min}$ & minimum zoom-out probability & 0.01 \\
    $\rho$ & zoom-in factor & 0.4 \\
    $r$ & resolution parameter for restart & 0.01 \\
    $C_\text{fail}$ & critical value for number of consecutive failures & $\text{max}(\lceil d/N_\text{par} \rceil, 2)$\\
    $\Delta\gamma$ & change of $\gamma$ value & 2 \\
    \bottomrule
  \end{tabular}
\end{table}

For test problems we used a suite of standard optimization benchmark problems from the literature and two hyperparameter-tuning problems: (1) tuning 5 hyperparameters of a random forest (2) tuning 7 hyperparameters of a deep neural network. The details of the benchmark problems and the hyperparameter-tuning problems are given in Appendix~\ref{sect:benchmark} and Appendix~\ref{sect:hyper} respectively.

\subsection{Optimization performance versus iteration}\label{subsec:optim_perform}

The first results that we consider are the optimization result (function value) versus iteration number. All the algorithms are proposing and evaluating the same number of points per iteration, so these results measure the quality of these proposed points. As we will see, ProSRS does significantly better than the existing methods.

Figure~\ref{fig:benchmark_opt_curve} shows performance of the algorithms on 12 optimization benchmark functions, for which the dimension varies from 2 to 10 (the last numeric figure in the function name indicates the dimension). The objective function on the y axis is the evaluation of the underlying true expectation function (not the noisy function) at the algorithm output. The error bar shows the standard deviation of 20 independent runs. All algorithms are run using 12 parallel cores of a Blue Waters\footnote{Blue Waters: \url{https://bluewaters.ncsa.illinois.edu}.}  XE compute node.

As we can see from Fig.~\ref{fig:benchmark_opt_curve}, our algorithm performs the best on almost all of the problems. In particular, ProSRS is significantly better on high-dimensional functions such as Ackley and Levy, as well as highly-complex functions such as Dropwave and Schaffer. Excellent performance on these benchmark problems shows that our algorithm can cope with various optimization landscape types.

\begin{figure}
\begin{center}
\includegraphics[width=0.9\textwidth]{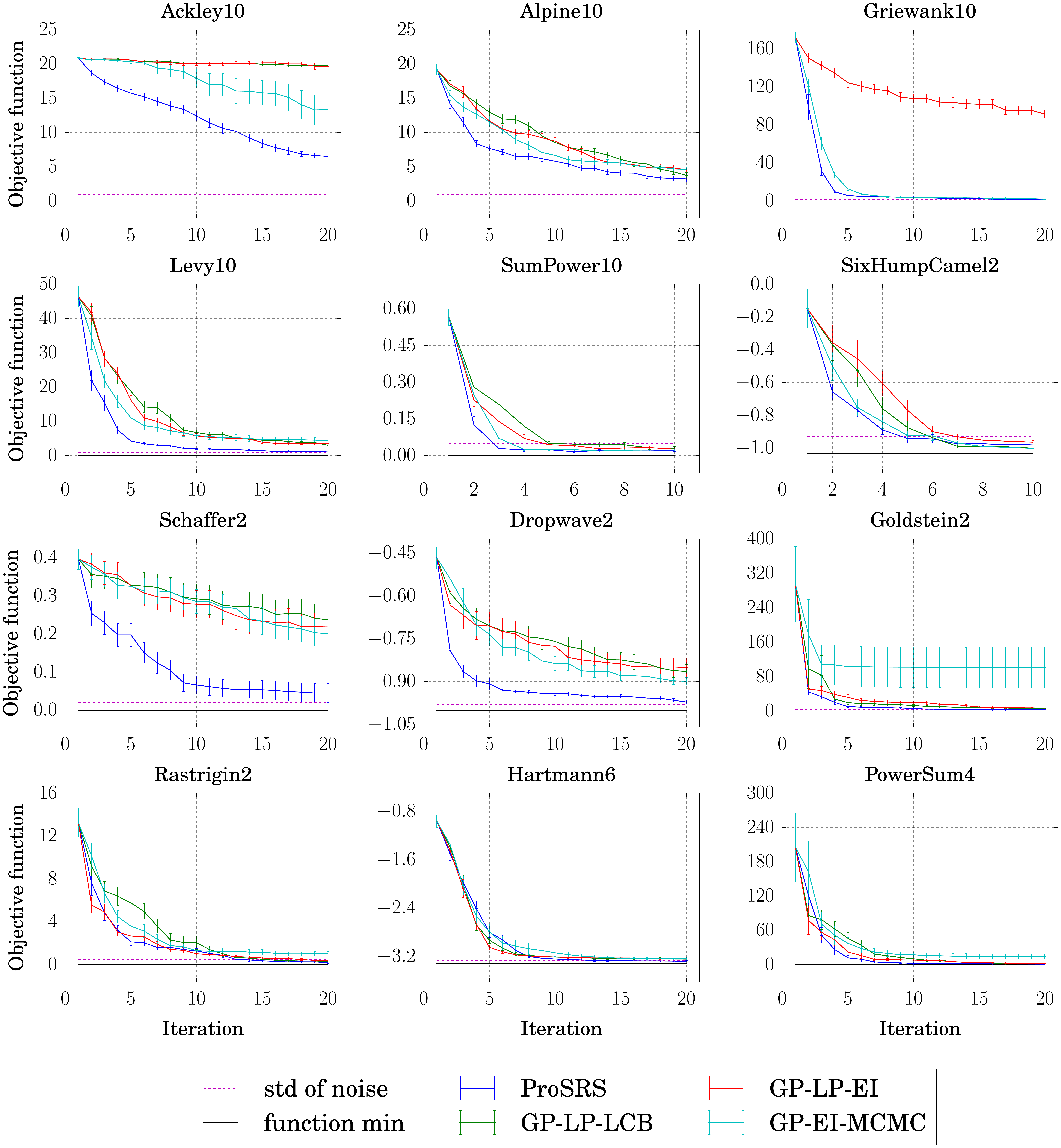}
\caption{Optimization curves for the benchmark functions. The error bar shows the standard deviation of 20 independent runs.}
\label{fig:benchmark_opt_curve}
\end{center}
\end{figure}


Figure~\ref{fig:hyperparameter_optim_curve} shows optimization performance on the two hyperparameter-tuning problems. Here we also include the random search algorithm since random search is one of the popular hyperparameter-tuning methods \citep{bergstra2012random}. First, we see that surrogate optimization algorithms are in general significantly better than the random search algorithm. This is no surprise as the surrogate optimization algorithm selects every evaluation point carefully in each iteration. Second, among the surrogate optimization algorithms, our ProSRS algorithm is better than the GP-EI-MCMC algorithm (particularly on the random forest tuning problem), and is much better than the two GP-LP algorithms.

\begin{figure}
\begin{center}
\includegraphics[width=0.9\textwidth]{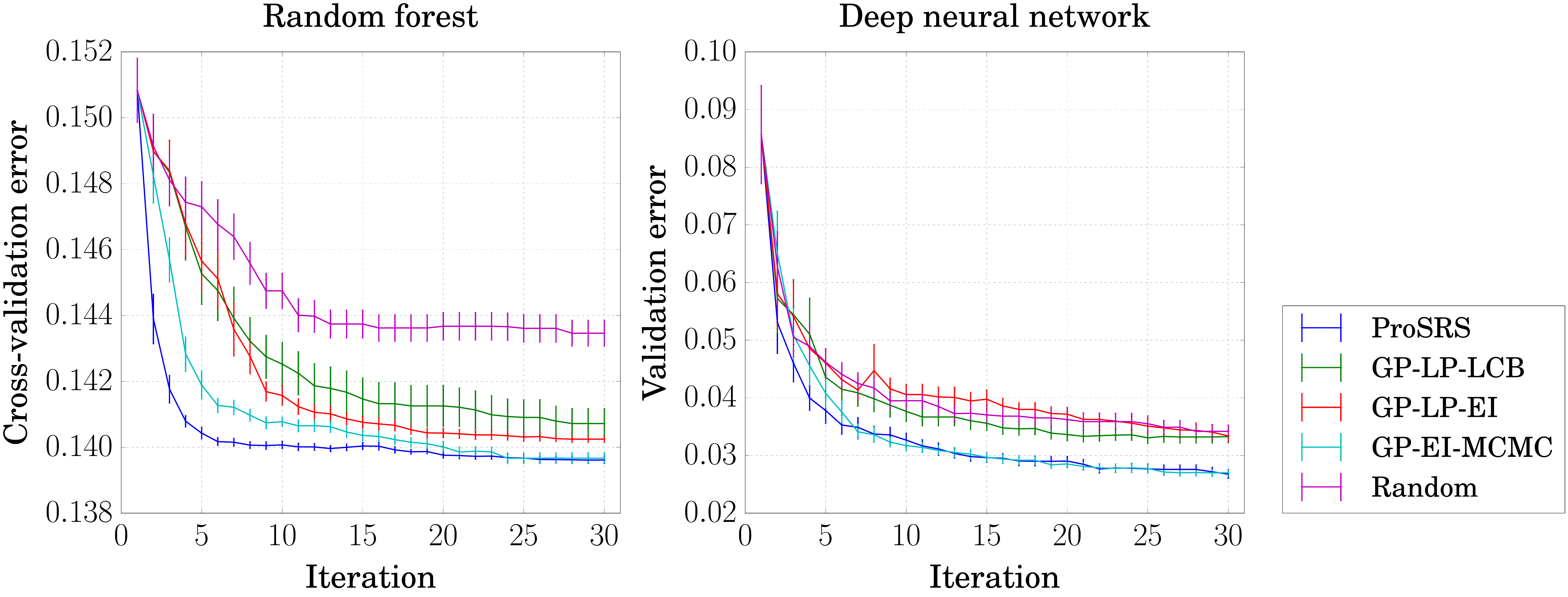}
\caption{Optimization curves for the hyperparameter-tuning problems. The error bar shows the standard deviation of 20 independent runs. All algorithms were run using 8 parallel cores of a Blue Waters XE compute node. The expected error (objective $F$) was estimated by averaging 5 independent samples.}
\label{fig:hyperparameter_optim_curve}
\end{center}
\end{figure}

\subsection{Optimization performance analysis}\label{subsec:optim_perform_analysis}

In Section~\ref{subsec:optim_perform} we demonstrated that our ProSRS algorithm generally achieved superior optimization performances compared to the Bayesian optimization algorithms. In this section, we give some insight into why our algorithm could be better. We performed the analysis with a numerical experiment that studied the modeling capability of RBF (as used in ProSRS) and GP models (as used in the Bayesian optimization methods). 

More specifically, we investigated RBF and GP regression on the twelve optimization benchmark functions that are shown in Fig.~\ref{fig:benchmark_opt_curve}, varying the number $n$ of training data points from 10 to 100. For each test function and every $n$, we first randomly sampled $n$ points $(X_1, X_2,\ldots, X_n)$ over the function domain using Latin hypercube sampling, and then evaluated these $n$ sampled points to get noisy responses $(Y_1, Y_2, \ldots, Y_n)$. Then given the data $(X_1, Y_1)$, $(X_2, Y_2),\ldots$, $(X_n, Y_n)$, we trained 4 models: a RBF model using the cross validation procedure developed in the ProSRS algorithm with no weighting, and 3 GP models with commonly used GP kernels: Matern1.5, Matern2.5 and RBF.

We used the Python scikit-learn package\footnote{Python package for Gaussian Processes: \url{http://scikit-learn.org/stable/modules/gaussian_process.html}.} for the implementations of GP regression. We set the number of restarts for the optimizer in GP regression to be 10.
We evaluated each regression model by measuring the relative error in terms of the $\text{L}_2$ norm of the difference between a model $g$ and the underlying true function $\mathbb{E}[f]$ over the function domain. We repeated the training and evaluation procedure for 10 times, and reported the mean and the standard deviation of the measured relative errors.

The results are shown in Fig.~\ref{fig:surrogate_performance}. We can see that cross-validated RBF regression (as used in our ProSRS method) generally produces a better model than those from GP regression (as used in the Bayesian optimization methods). Specifically, the RBF model from ProSRS is significantly better for the test functions Griewank, Levy, Goldstein and PowerSum, and is on par with GP models for Schaffer, Dropwave and Hartmann.

\begin{figure}
\begin{center}
\includegraphics[width=0.9\textwidth]{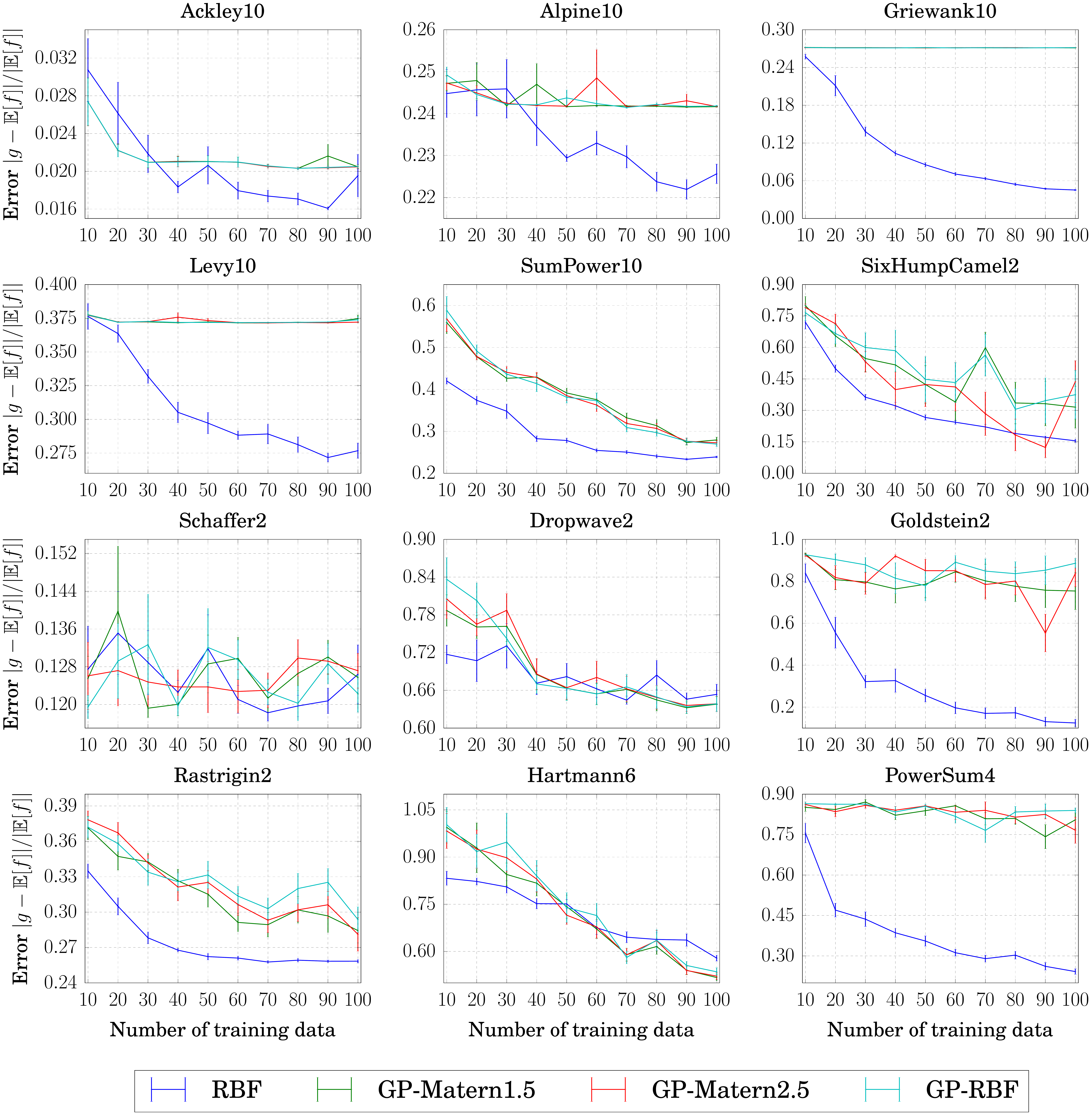}
\caption{Compare the modeling capability of RBF regression as used in ProSRS (dark blue lines) and GP regression with kernels: Matern1.5, Matern2.5 and RBF (green, red and light blue lines respectively) on 12 optimization benchmark functions. The y axis is the relative error in terms of the $\text{L}_2$ norm of the difference between a model $g$ and the underlying true function $\mathbb{E}[f]$ over the function domain. The error bar shows the standard deviation of 10 independent runs.}
\label{fig:surrogate_performance}
\end{center}
\end{figure}

From this numerical study, we can draw two conclusions. First, the ProSRS RBF models seem to be able to better capture the objective functions than GP regression models. One possible explanation for this is that the ProSRS RBF regression uses a cross validation procedure so that the best model is selected directly according to the data, whereas GP regression builds models relying on the prior distributional assumptions about the data (i.e., Gaussian process with some kernel). Therefore, in a way the ProSRS regression procedure makes fewer assumptions about the data and is more ``data-driven'' than GP. Since the quality of a surrogate has a direct impact on how well the proposed points exploit the objective function, we believe that the superiority of the RBF models plays an important part in the success of our ProSRS algorithm over those Bayesian optimization algorithms.

Second, for those test functions where ProSRS RBF and GP have similar modeling performances (i.e., Schaffer, Dropwave and Hartmann), the optimization performance of ProSRS (using RBF) is nonetheless generally better than Bayesian optimization (using the GP models), as we can see from Fig.~\ref{fig:benchmark_opt_curve}. This suggests that with surrogate modeling performance being equal, the ProSRS point selection strategy (i.e., SRS and zoom strategy) may still have an edge over the probablity-based selection criterion (e.g., EI-MCMC) of Bayesian optimization.

\subsection{Algorithm cost}\label{subsec:optim_alg_cost}

In Section~\ref{subsec:optim_perform} we saw that ProSRS achieved faster convergence per iteration, meaning that it was proposing better points to evaluate in each iteration. In this section we will compare the cost of the algorithms and show that ProSRS is in addition much cheaper per iteration. The main focus here is to compare the cost of the algorithm, not the cost of evaluating the function $f$ since the function-evaluation cost is roughly the same among the algorithms.

Figure~\ref{fig:alg_cost_benchmark} and Figure~\ref{fig:alg_cost_hyperparam} show the computational costs of running different algorithms for the twelve optimization benchmark problems and the two hyperparameter-tuning problems. The time was benchmarked on Blue Waters XE compute nodes. We observe that  our ProSRS algorithm is about 1--2 orders of magnitude cheaper than the two GP-LP algorithms, and about 3--4 orders of magnitude cheaper than the GP-EI-MCMC algorithm. It is worth noting that for the hyperparameter-tuning problems (Fig.~\ref{fig:alg_cost_hyperparam}), the cost of the GP-EI-MCMC algorithm is in fact consistently higher than that of the training and the evaluation of a machine learning model, and the cost gap becomes larger as the number of iterations increases.

From Fig.~\ref{fig:alg_cost_hyperparam} we can also see that the cost of our algorithm scales roughly $\sim\mathcal{O}(1)$ with the number of iterations in the long run (i.e., when ProSRS is run with a large number of iterations, the general trend of the cost stays flat with iterations). This scaling behavior is generally true for our algorithm, and is a consequence of the zoom strategy and the restart mechanism exploited by our algorithm.

\begin{figure}
\begin{center}
\includegraphics[width=0.9\textwidth]{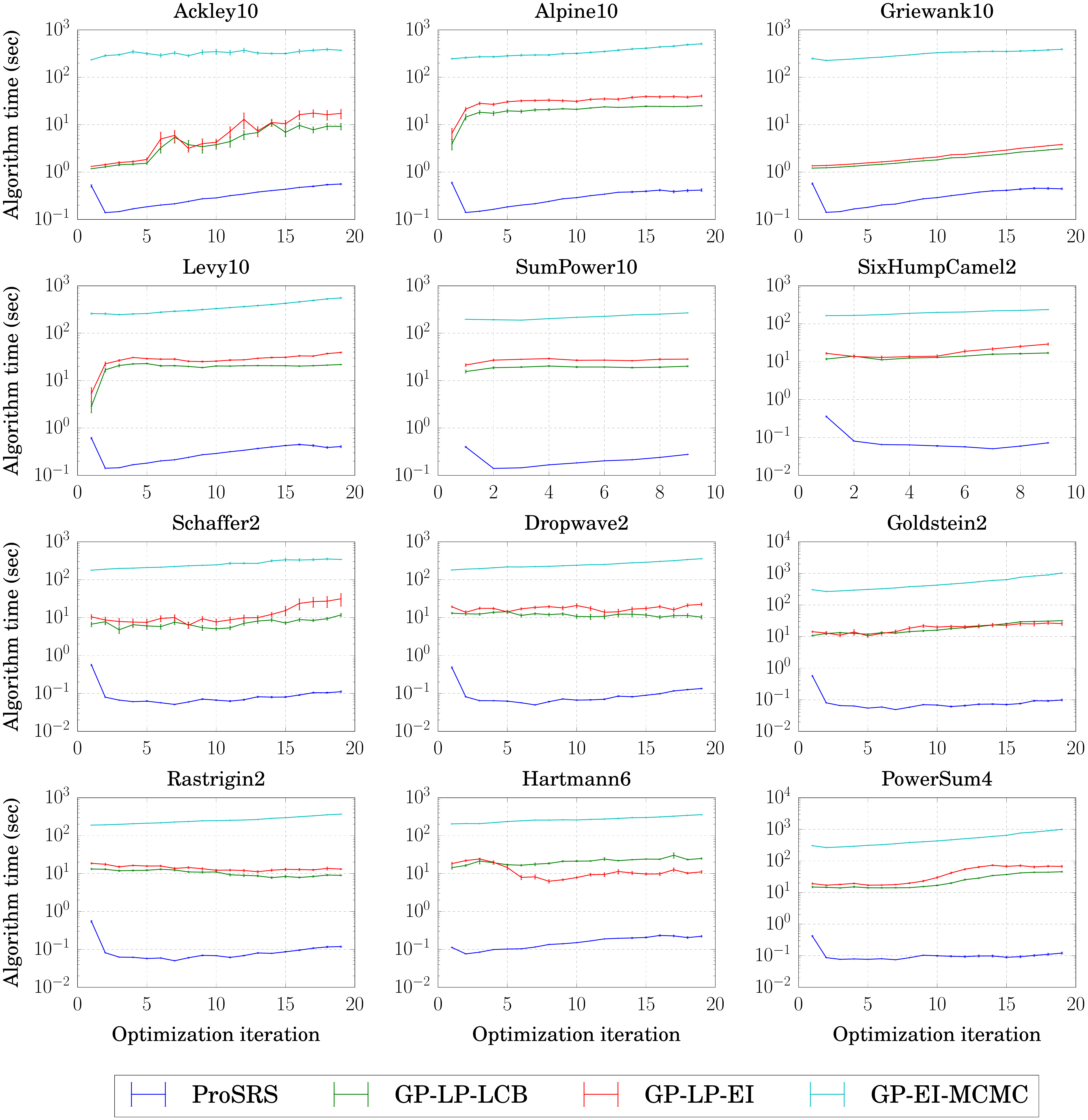}
\caption{Computational costs of different algorithms for the twelve optimization benchmark problems. The plots show the mean and standard deviation of 20 independent runs. The x axis is the number of iterations in actual optimization excluding the initial DOE iteration. The y axis is the actual time that was consumed by an algorithm in each iteration, and does not include the time of parallel function evaluations.}
\label{fig:alg_cost_benchmark}
\end{center}
\end{figure}

\begin{figure}
\begin{center}
\includegraphics[width=0.85\textwidth]{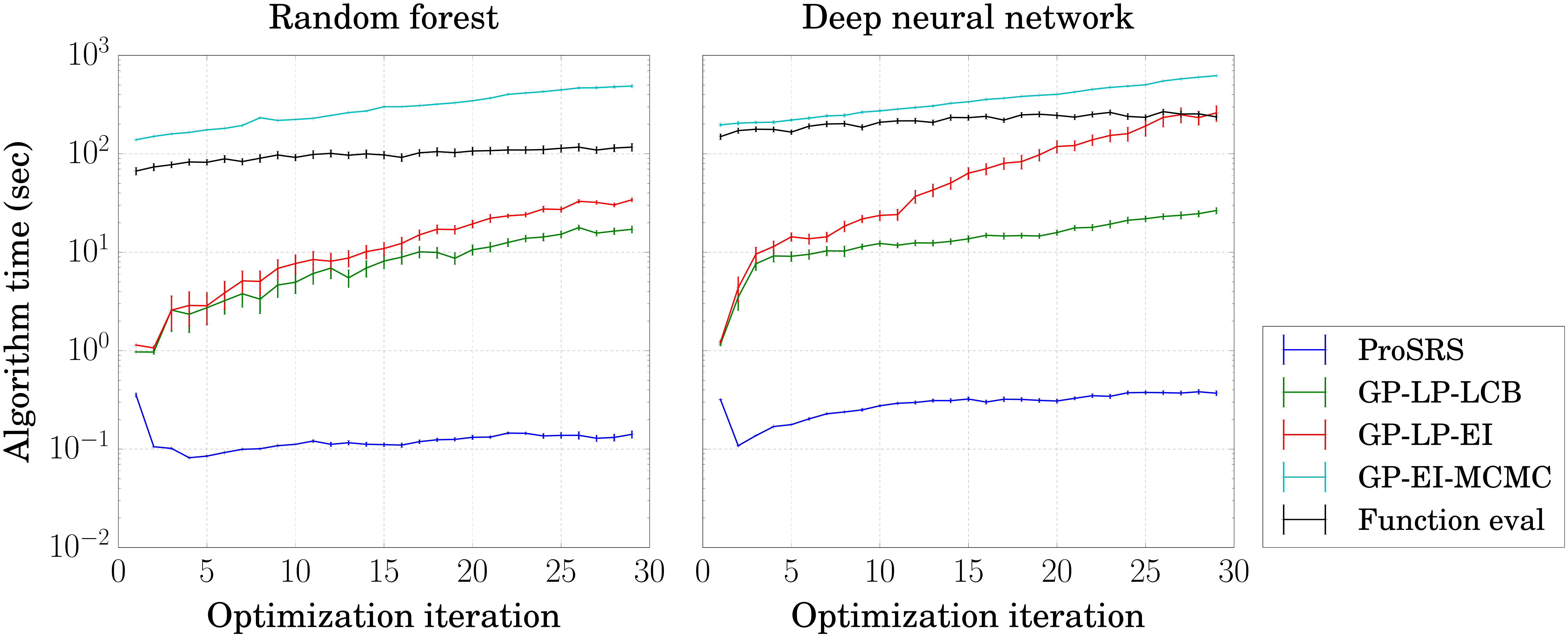}
\caption{Computational costs of different algorithms for the two hyperparameter tuning problems. The plots show the mean and standard deviation of 20 independent runs. The x axis is the number of iterations in actual optimization excluding the initial DOE iteration. For different algorithms, the y axis is the actual time that was consumed by the algorithm in each iteration, and does not include the time of parallel function evaluations. The time for training and evaluating the machine learning models is shown in black.}
\label{fig:alg_cost_hyperparam}
\end{center}
\end{figure}

\subsection{Overall optimization efficiency}\label{subsec:optim_eff}

In this section, we will show the overall optimization efficiency for the two hyperparameter-tuning problems, which takes into account not only the optimization performance per iteration but also the cost of the algorithm and the expensive function evaluations. From Fig.~\ref{fig:hyperparameter_optim_efficiency}, we can see that our ProSRS algorithm is the best among all the algorithms. Because of the high cost of the GP-EI-MCMC algorithm, the advantage of our algorithm over GP-EI-MCMC becomes even more pronounced compared to that of the iteration-based performance measurement (Fig.~\ref{fig:hyperparameter_optim_curve}).

\begin{figure}
\begin{center}
\includegraphics[width=0.85\textwidth]{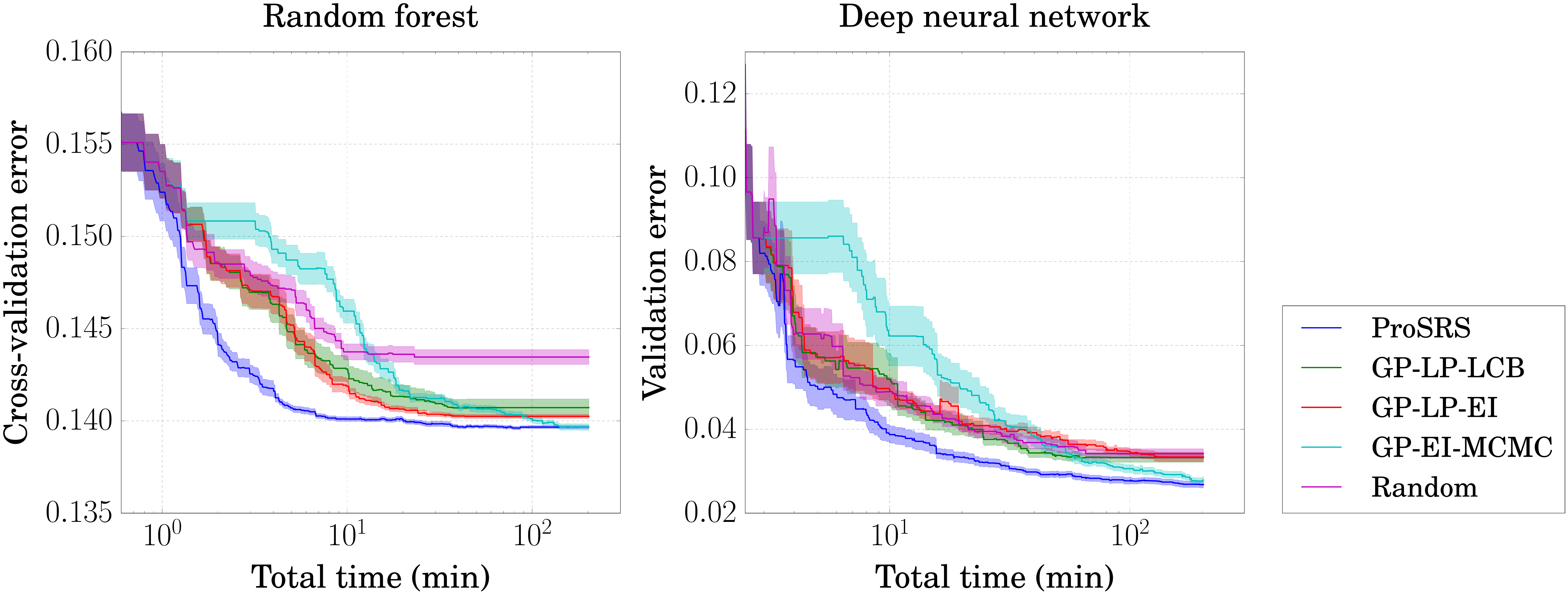}
\caption{Optimization efficiency of different algorithms on the two hyperparameter-tuning problems. Total time on the horizontal axis is the actual elapsed time including both algorithm running time and time of evaluating expensive functions. The shaded areas show the standard deviation of 20 independent runs.}
\label{fig:hyperparameter_optim_efficiency}
\end{center}
\end{figure}

\section{Conclusion}\label{sect:conclusion}

In this paper we introduced a novel parallel surrogate optimization algorithm, ProSRS, for noisy expensive optimization problems. We developed a ``zoom strategy'' for efficiency improvement, a weighted radial basis regression procedure, and a new SRS method combining the two types of candidate points in the original SRS work. We proved an analytical result for our algorithm (Theorem~\ref{thm:converge}): if ProSRS is run for sufficiently long, with probability converging to one there will be at least one point among all the evaluations that will be arbitrarily close to the global minimizer of the objective function.  Numerical experiments show that our algorithm outperforms three current Bayesian optimization algorithms on both optimization benchmark problems and real machine learning hyperparameter-tuning problems. Our algorithm not only shows better optimization performance per iteration but also is orders of magnitude cheaper to run.

\section{Acknowledgments}
This research was supported by NSF CMMI-1150490 and is part of the Blue Waters sustained-petascale computing project, which is supported by the National Science Foundation (awards OCI-0725070 and ACI-1238993) and the state of Illinois. Blue Waters is a joint effort of the University of Illinois at Urbana-Champaign and its National Center for Supercomputing Applications.

\bibliography{refs}


\begin{appendices}

\section{Proof of Theorem~\ref{thm:converge}}\label{sect:converge_proof}

\begin{thm}
Suppose the objective function $F$ in Eq.~\ref{eq:optim_prob} is continuous on the domain $\mathcal{D}\subseteq\mathbb{R}^d$ and $x_\text{opt}$ is the unique minimizer of $F$, characterized by\footnote{Here we adopt the convention that if $\{x\in\mathcal{D}, \lVert x-x_\text{opt} \rVert \geq \eta\}= \emptyset$, then $\inf_{x\in\mathcal{D}, \lVert x-x_\text{opt} \rVert \geq \eta} F(x) = +\infty$.} $F(x_\text{opt}) = \inf_{x\in\mathcal{D}} F(x) \in (-\infty, +\infty)$ and $\inf_{x\in\mathcal{D}, \lVert x-x_\text{opt} \rVert \geq \eta} F(x) > F(x_\text{opt})$ for all $\eta>0$. Let $x_n$ be the point with the minimum objective value among all the evaluated points up to iteration $n$. Then $x_n \xrightarrow{} x_\text{opt}$ almost surely as $n\to\infty$. 
\end{thm}

\begin{proof}

We define the zoom level $z$ to be zero for the root node and, whenever zooming in occurs, the zoom level of the child node is one plus that of its parent node so that every node in the tree is associated with a unique zoom level (see Fig.~\ref{fig:tree_diagram}). 

First, we argue that there is an upper bound on the zoom level for ProSRS algorithm. Since after each zoom-in step, the size of the domain is shrunk by at least the zoom-in factor $\rho\in(0,1)$, the domain length for a node of a zoom level $z\in\mathbb{N}$ is upper bounded by $\rho^z(b_i-a_i)$ for each dimension $i=1,2,\dots,d$. Here $a_i$ and $b_i$ are the domain boundaries for the root node (Eq.~\ref{eq:optim_prob}). Now let us consider a node with zoom level $z^* = \lceil \log_\rho r \rceil + 1$, where $r\in(0,1)$ is the prescribed resolution parameter for the restart (see Eq.~\ref{eq:restart_cond}). We further denote the domain length of this node in each dimension to be $\ell_i$ and the number of evaluation points within its domain to be $n$, then we have for all $i=1,2,\dots,d$,
\begin{equation*}
	n^{-\frac{1}{d}}\ell_i \leq \ell_i \leq \rho^{z^*}(b_i-a_i) = \rho^{\lceil \log_\rho r \rceil + 1}(b_i-a_i) < \rho^{\log_\rho r}(b_i-a_i) = r(b_i-a_i),
\end{equation*}
which would satisfy the restart condition (Eq.~\ref{eq:restart_cond}). This implies that the zoom level of ProSRS must be less than $z^*$. In other words, the zoom level is upper bounded by $z_\text{max} = z^*-1 = \lceil \log_\rho r \rceil$.

Now fix some $\epsilon>0$ and define $\Delta \coloneqq \text{max}(z_\text{max}, N_\text{DOE}+1)$, where $N_\text{DOE}$ is the number of iterations for the initial space-filling design. The main idea of the following proof is similar to that in the original SRS paper \citep{regis2007stochastic}.

Since the objective function $F$ is continuous at the unique minimizer $x_\text{opt}$, there exists $\delta(\epsilon)>0$ so that whenever $x$ is within the open ball $\mathcal{B}(x_\text{opt}, \delta(\epsilon))$, $f(x) < f(x_\text{opt})+\epsilon$. 

The probability that a candidate point generated in the root node (of either Type~I or Type~II) is located within the domain $\mathcal{B}(x_\text{opt}, \delta(\epsilon)) \cap \mathcal{D}$ can be shown to be bounded from below by some positive $\nu(\epsilon)$ (see Section 2 of \citet{regis2007stochastic}). Here $\mathcal{D}$ is the domain of the optimization problem (Eq.~\ref{eq:optim_prob}). Since all the candidate points are generated independently, the probability that all the candidate points are within $\mathcal{B}(x_\text{opt}, \delta(\epsilon)) \cap \mathcal{D}$ is greater than or equal to $L(\epsilon) \coloneqq \nu(\epsilon)^t > 0$, where $t$ is a constant denoting the number of candidate points.

Now we define a positive quantity $h(\epsilon) \coloneqq L(\epsilon)(\beta_\text{min})^\Delta$, where $\beta_\text{min}$ is the minimum zoom-out probability (see Section~\ref{sub_sect:zoom}). We further define the event
\begin{align*}
A_i \coloneqq &\{\text{for each of the iterations } (i-1)\Delta+1, (i-1)\Delta+2, \ldots, i\Delta, \text{ there is at least one}\\
 & \quad \text{candidate point that lies outside the domain }\mathcal{B}(x_\text{opt}, \delta(\epsilon)) \cap \mathcal{D}\}, \quad i \in\mathbb{Z}^+.
\end{align*}
Let probability $P_i \coloneqq P(A_i\mid A_1\cap A_2\cap\ldots\cap A_{i-1})$ with the understanding that $P_1 = P(A_1)$. Then
\begin{equation}\label{eq:prob_expression}
	P(A_1\cap A_2\cap\ldots\cap A_k) = \prod_{i=1}^k P_i, \quad k\in\mathbb{Z}^+.
\end{equation}

For now, let us assume $i>1$. For the iteration $(i-1)\Delta$, there are 3 possible events that could happen when we are about to run Line~\ref{pseudo_code:check_zoom_out} of the ProSRS algorithm (Alg.~\ref{alg:ProSRS}): 
\begin{align*}
E_1 &= \{\text{decide to restart}\}, \\
E_2 &= \{\text{decide not to restart and the parent node exists}\},\\
E_3 &= \{\text{decide not to restart and the parent node does not exist}\}.
\end{align*}
Let $z_{i-1}$ be the zoom level of the current node at this moment. Then we have the following inequalities:

\begin{align*}
\begin{split}
		&P(\overline{A_i}\mid A_1\cap A_2\cap\ldots\cap A_{i-1}\cap E_1) \\
	&= P(\text{among iterations }(i-1)\Delta+1, (i-1)\Delta+2, \ldots, i\Delta,\text{there exists one iteration for which}\\
	  &\quad \text{all the candidate points are within domain } \mathcal{B}(x_\text{opt}, \delta(\epsilon)) \cap \mathcal{D} \mid A_1\cap A_2\cap\ldots\cap A_{i-1}\cap E_1)\\
	&\geq P(\text{all the candidate points are within } \mathcal{B}(x_\text{opt}, \delta(\epsilon))\cap \mathcal{D}\text{ for iteration } \big((i-1)\Delta+N_\text{DOE}+1\big)\\
	     &\qquad \mid A_1\cap A_2\cap\ldots\cap A_{i-1}\cap E_1) \geq L(\epsilon) \geq h(\epsilon)
\end{split}\\
\begin{split}
		&P(\overline{A_i}\mid A_1\cap A_2\cap\ldots\cap A_{i-1}\cap E_2) \\
	&\geq P(\text{decide to zoom out for iterations } (i-1)\Delta, (i-1)\Delta+1, \ldots, (i-1)\Delta+z_{i-1}-1 \text{ and}\\
	&\qquad \text{all the candidate points are within } \mathcal{B}(x_\text{opt}, \delta(\epsilon))\cap \mathcal{D}\text{ for the iteration } \big((i-1)\Delta+z_{i-1}\big)\\
	     &\qquad \mid A_1\cap A_2\cap\ldots\cap A_{i-1}\cap E_2) \geq L(\epsilon)(\beta_\text{min})^{z_{i-1}} \geq h(\epsilon) 
\end{split}\\
		&P(\overline{A_i}\mid A_1\cap A_2\cap\ldots\cap A_{i-1}\cap E_3) \\
	&\geq P(\text{all the candidate points are within } \mathcal{B}(x_\text{opt}, \delta(\epsilon))\cap \mathcal{D}\text{ for the iteration } \big((i-1)\Delta+1\big)\\
	     &\qquad \mid A_1\cap A_2\cap\ldots\cap A_{i-1}\cap E_3) \geq L(\epsilon) \geq h(\epsilon).
\end{align*}

That is, for any $i>1$, $P(\overline{A_i}\mid A_1\cap A_2\cap\ldots\cap A_{i-1}\cap E_j) \geq h(\epsilon)$ for all $j = 1,2,3$. Hence, $P(\overline{A_i}\mid A_1\cap A_2\cap\ldots\cap A_{i-1})\geq h(\epsilon)$, which implies $P_i \leq 1-h(\epsilon)$ for any $i>1$. Now if $i = 1$, again we have $P_1 = 1-P(\overline{A_1})\leq 1-h(\epsilon)$ because the probability that all the candidates are within $\mathcal{B}(x_\text{opt}, \delta(\epsilon))\cap \mathcal{D}$ for the iteration $(N_\text{DOE}+1)$ is greater or equal to $h(\epsilon)$. Therefore, $P_i \leq 1-h(\epsilon)$ holds true for all $i\in\mathbb{Z}^+$. Using Eq.~\ref{eq:prob_expression}, we have
\begin{equation}\label{eq:prob_ineq}
	P(A_1\cap A_2\cap\ldots\cap A_k) \leq \big(1-h(\epsilon)\big)^k.
\end{equation}

Since $h(\epsilon)\in(0,1)$, $P(A_1\cap A_2\cap\ldots\cap A_k)$ converges to zero, or equivalently $P(\overline{A_1\cap A_2\cap\ldots\cap A_k})$ converges to one as $k\to\infty$. Observe that
\begin{align*}
&\overline{A_1\cap A_2\cap\ldots\cap A_k} \\
&= \{\text{among iterations }1,2,\ldots,k\Delta, \text{there is an iteration for which all the candidate points}\\
&\qquad \text{are within }\mathcal{B}(x_\text{opt}, \delta(\epsilon))\cap \mathcal{D}\}\\
&\subseteq \{\text{among iterations }1,2,\ldots,k\Delta, \text{there is an evaluated point } x \text{ within }\mathcal{B}(x_\text{opt}, \delta(\epsilon))\cap \mathcal{D}\}\\
&\subseteq \{\text{among iterations }1,2,\ldots,k\Delta, \text{there is an evaluated point } x \text{ such that } f(x)<f(x_\text{opt})+\epsilon\}\\
&\subseteq \{f(x_{k\Delta}) < f(x_\text{opt})+\epsilon\} = \{|f(x_{k\Delta})-f(x_\text{opt})|<\epsilon\}.
\end{align*} 
Hence, $f(x_{k\Delta})$ converges to $f(x_\text{opt})$ in probability as $k\to\infty$. Therefore, there is a subsequence of $\big(f(x_{k\Delta})\big)_{k\in\mathbb{N}}$ which is also a subsequence of $\big(f(x_n)\big)_{n\in\mathbb{N}}$, that converges almost surely to $f(x_\text{opt})$. Because $f(x_n)$ is monotonically decreasing so that the limit always exists, $f(x_n)$ converges to $f(x_\text{opt})$ almost surely. Finally, by the uniqueness of the minimizer, $x_n$ converges to $x_\text{opt}$ almost surely. The arguments for the last two almost-sure convergences are essentially the same as those used in proving the convergence of a simple random search algorithm (see the proof of Theorem 2.1 in \citet{spall2005introduction}). 

\end{proof}

\section{Optimization benchmark functions}\label{sect:benchmark}

Table~\ref{table:benchmark_func} summarizes the benchmark test problems. For each problem, a Gaussian noise was added to the true underlying function. We tested with commonly-used optimization domains, and the standard deviation of the noise roughly matched the range of a function.

\begin{table}
  \caption{Optimization benchmark problems (the last numeric figure in the function name indicates the dimension of the problem)}
  \label{table:benchmark_func}
  \centering
  \begin{tabular}{ccc}
    	\toprule
    Function  &  Optimization Domain &  Std. of Gaussian noise \\
    \midrule
    Ackley10 & $[-32.768,32.768]^{10}$ &  1\\
	Alpine10 & $[-10,10]^{10}$ & 1\\
	Griewank10 & $[-600,600]^{10}$ & 2\\
	Levy10 & $[-10,10]^{10}$ & 1\\
	SumPower10 & $[-1,1]^{10}$ & 0.05\\
	SixHumpCamel2 & $[-3,3]\times[-2,2]$ & 0.1\\
	Schaffer2 & $[-100,100]^2$ & 0.02\\
	Dropwave2 & $[-5.12,5.12]^2$ & 0.02\\
	Goldstein-Price2 & $[-2,2]^2$ & 2\\
	Rastrigin2 & $[-5.12,5.12]^2$ & 0.5\\
	Hartmann6 & $[0,1]^6$ & 0.05\\
	PowerSum4 & $[0,4]^4$ & 1\\
    \bottomrule
  \end{tabular}
\end{table}

\section{Hyperparameter-tuning problems}\label{sect:hyper}

Hyperparameter tuning can be formulated as an optimization problem in Eq.~\ref{eq:optim_prob}. In this case, the function $f$ is a validation or cross-validation error for a machine learning model and the vector $x$ represents the hyperparameters to be tuned. The function $f$ is typically expensive since one evaluation of $f$ involves training and scoring one or multiple machine learning models. The noise associated with $f$ may come from the fact that a machine learning algorithm (e.g., random forest) contains random elements or a stochastic optimization method (e.g., SGD) is invoked during the training process. 

Next, we describe the details of the two hyperparameter-tuning problems considered in this work. For both problems, when tuning an integer-valued hyperparameter, we rounded the continuous output from an optimization algorithm to the nearest integer before feeding it to the machine learning algorithm.

\paragraph{Random forest.} 

We tuned a random forest, one of the most widely used classification algorithms, on the well-known Adult dataset \citep{Dua:2017}. The dataset consists of 48842 instances with 14 attributes, and the task is to classify income based on census information. We tuned 5 hyperparameters: number of trees on $[1,300]$, number of features on $[1,14]$, maximum depth of a tree on $[1,100]$, minimum number of samples for the node split on $[2,1000]$ and minimum number of samples for a leaf node on $[1,1000]$. We minimized the 5-fold cross-validation error.

\paragraph{Deep neural network.} 

We tuned a feedforward deep neural network with 2 hidden layers on the popular MNIST dataset \citep{mnist}. This tuning problem is also considered in \citep{NIPS2017_7111}. We used the same training-validation data split as in the TensorFlow tutorial \citep{tensorflow} with the training set having 55000 data points and the validation set having 5000 data points. We tuned 7 hyperparameters: number of units in each hidden layer on $[1,100]$, $\text{L}_1$ and $\text{L}_2$ regularization constants, both on a log scale on $[10^{-8}, 10^0]$, learning rate on a log scale on $[10^{-4}, 10^0]$, batch size on $[50,1000]$ and number of epochs on $[5,50]$. We minimized the validation error.

\end{appendices}

\end{document}